\documentclass[10pt]{amsart}
\usepackage{amsmath,amssymb,amsthm,epsfig}
\usepackage[all]{xy}
\usepackage{color}
\usepackage{graphicx}
\usepackage{labelfig}

\newtheorem{theorem}{Theorem}[section]
\newtheorem{thm}[theorem]{Theorem}
\newtheorem{lemma}[theorem]{Lemma}
\newtheorem{lem}[theorem]{Lemma}

\newtheorem{prop}[theorem]{Proposition}

\theoremstyle{definition}

\newtheorem{defn}[theorem]{Definition}
\newtheorem{remark}[theorem]{Remark}
\newtheorem{rem}[theorem]{Remark}

\newtheorem{exam}[theorem]{Example}

\def\co{\colon \thinspace}

\newcommand{\GL}{\textrm{GL}}
\newcommand{\Z}{\mathbb{Z}}
\newcommand{\R}{\mathbb{R}}

\newcommand{\C}{\mathbb{C}}
\newcommand{\mL}{\mathcal{L}}
\newcommand{\tC}{\widetilde{C}}
\newcommand{\tSigma}{\widetilde{\Sigma}}

\newcommand{\LWcr}{\textsf{LWcr}}
\newcommand{\SWcr}{\textsf{SWcr}}
\newcommand{\Wcr}{\textsf{Wcr}}

\title[Linear representations and dual Garside length of braids]{Lawrence-Krammer-Bigelow representations and dual Garside length of braids}
\author{Tetsuya Ito}
\address{Graduate School of Mathematical Science, University of Tokyo, 3-8-1 Komaba Meguro-ku Tokyo 153-8914, Japan}
\email{tetitoh@ms.u-tokyo.ac.jp}
\urladdr{http://ms.u-tokyo.ac.jp/~tetitoh/}
\author{Bert Wiest}
\address{IRMAR, UMR 6625 du CNRS, Universit\'e de Rennes 1, Campus de Beaulieu, 35042 Rennes Cedex, France}
\email{bertold.wiest@univ-rennes1.fr}
\urladdr{http://perso.univ-rennes1.fr/bertold.wiest}

\begin{document}

\begin{abstract} 
We show that the span of the variable $q$ in the Lawrence-Krammer-Bigelow representation matrix of a braid is equal to the twice of the dual Garside length of the braid, as was conjectured by Krammer.
Our proof is close in spirit to Bigelow's geometric approach. The key observation is that the dual Garside length of a braid can be read off a certain labeling of its curve diagram. 
\end{abstract}
 \maketitle

\section{Introduction}

The question whether the braid group $B_{n}$ is linear or not was a long standing problem.
At the end of the 20th century, the problem was solved affirmatively by Krammer \cite{k2} and Bigelow \cite{b2} independently.
They showed that a certain linear representation of the braid group first constructed by Lawrence \cite{l} and now called the {\em Lawrence-Krammer-Bigelow representation} (LKB representation, for short) is faithful.
Interestingly, the two proofs of the faithfulness of the LKB representations are completely different: Bigelow's proof is geometric whereas Krammer's proof is algebraic.

Our main result is (Theorem~\ref{thm:q}) that the span of the variable $q$ in the image of the LKB representation of a braid~$\beta$ is equal to twice the dual Garside length of~$\beta$, as conjectured by Krammer in \cite{k1}. 
This is an analogue of Krammer's theorem \cite{k2} that the span of the variable $t$ in the image of the LKB representation is equal to the classical Garside length. It is remarkable that, despite the analogy with Krammer's result, our proof is rather based on Bigelow's techniques.

One of our main tools is a labeling of curve diagrams of braids called the {\em wall crossing labeling}. In Theorem \ref{thm:wall}, we will prove that the wall crossing labeling tells us the dual Garside length of braids. On the other hand, in Lemma~\ref{lem:Wcr=q} we will observe that the wall crossing labeling is also related to the variable $q$ in the noodle-fork pairing, a homological intersection pairing appearing in the LKB representation. 
Thus wall crossing labelings of curve diagrams serve as a bridge connecting two seemingly unrelated objects, namely the LKB representation and the dual Garside structure.

In order to state our main Theorem, we set up some notation.
In this paper we denote the LKB representation by 
\[ \mL\co B_{n} \rightarrow \GL \left(\textstyle{\frac{n(n-1)}{2}}; \Z[q^{\pm 1},t^{\pm 1}]\right) \] 
The matrix representative $\mL(\beta)$ depends on the choice of the basis, and there are variations of the matrices of the LKB representation: all matrices given in the three papers \cite{b2},\cite{k1},\cite{k2} are different!

In this paper we use Bigelow's expression of the matrices given in \cite[Theorem~4.1]{b2}, with the correction of a sign error from \cite{b3}. This expression has geometric origin, and is essentially the same as the matrices in \cite{k1}, except that the variable $t$ of \cite{k1} corresponds to $-t$ of \cite{b2}.

Let us take the basis $\{F_{i,j}\}_{1\leq i<j \leq n}$ of $\R^{n(n-1)/2}$, defined by the {\em standard forks} -- see Section 2 for details. 
Using this basis, the matrix representative of the LKB representative is given as follows.
\[
[\mL(\sigma_{i})](F_{j,k}) = 
\left\{ \begin{array}{ll}
F_{j,k} & i \not \in \{ j-1,j,k-1,k \}\\
qF_{i,k} + (q^{2}-q)F_{i,j} + (1-q)F_{j,k} & i=j-1\\
F_{j+1,k} & i=j\neq k-1\\
qF_{j,i} + (1-q)F_{j,k} +(q-q^{2})tF_{i,k} & i=k-1\neq j\\
F_{j,k+1} & i=k\\
-q^{2}tF_{j,k} & i=j=k-1
\end{array}
\right.
\]
The image $\mL(\beta)$ is a $\frac{n(n-1)}{2} \times \frac{n(n-1)}{2}$-matrix whose entries are in $\Z[q^{\pm 1},t^{\pm 1}]$.
For a Laurent polynomial $a \in \Z[q^{\pm 1}, t^{\pm 1}]$, let $M_{q}(a)$ and $m_{q}(a)$ be the maximal and the minimal degree of the variable $q$, respectively.
For a matrix $A =(a_{\iota,\kappa}) \in \GL (\frac{n(n-1)}{2}; \Z[q^{\pm 1},t^{\pm 1}])$ we define
\[ M_{q}(A)=\max_{1\leqslant \iota,\kappa\leqslant \frac{n(n-1)}{2}}\{M_{q}(a_{\iota,\kappa})\}, \;\;\;\; m_{q}(A) = \min_{1\leqslant \iota,\kappa\leqslant \frac{n(n-1)}{2}}\{m_{q}(a_{\iota,\kappa})\} . \] 

We denote the supremum, the infimum and the length function of a braid~$\beta$ for the dual Garside structure by $\inf_{\Sigma^{*}}(\beta)$, $\sup_{\Sigma^{*}}(\beta)$, and $l_{\Sigma^{*}}(\beta)$, respectively -- see Section \ref{sec:Garside} for precise definitions.

The main result in this paper is the following, which was already conjectured by Krammer in \cite{k1}.

\begin{thm}[Variable $q$ in the LKB representation and the dual Garside length]
\label{thm:q}
For $\beta \in B_{n}$,
\begin{enumerate}
\item $2\sup_{\Sigma^{*}}(\beta) = M_{q}(\mL(\beta))$.
\item $2\inf_{\Sigma^{*}}(\beta) = m_{q}(\mL(\beta))$.
\item $2l_{\Sigma^{*}}(\beta) = \max\{0,M_{q}(\mL(\beta))\} - \min\{0,m_{q}(\mL(\beta))\}$.
\end{enumerate}
\end{thm}

\section*{Acknowledgement}

The first author was supported by the JSPS Institutional Program for Young Researcher Overseas Visits.

\section{Preliminaries}
First of all we set up our notation and conventions.
Let $D^{2}=\{z \in \C \: | \: |z| \leqslant 1\}$ be the unit disc in the complex plane and $D_{n} = D^{2} - \{p_{1},\ldots,p_{n}\}$ be the $n$-punctured disc, where each puncture $p_{i}$ is put on the real line so that $-1 < p_{1}<p_{2}< \cdots <p_{n} < 1$ holds. The braid group $B_{n}$ is identified with the mapping class group of~$D_{n}$.

Thoroughout this paper, we adopt the following conventions. Braids acts on the {\em left}. A positive standard generator $\sigma_{i}$ is identified with the {\em left-handed}, that is, the {\em clockwise} half Dehn-twist which interchanges the punctures $p_{i}$ and $p_{i+1}$.
Also, when calculating winding numbers, the positive direction of winding is taken as the {\em clockwise direction}. This convention is the opposite to Bigelow's.

\subsection{The dual Garside length}
\label{sec:Garside}

First we review the definition of the supremum, the infimum and the length functions of the dual Garside structure. We do not need detailed Garside theory such as normal forms or algorithms. We only need a length formula, Proposition \ref{prop:dGlength} below.

Recall that the braid group $B_{n}$ is presented as
\[ B_{n} = 
\left \langle 
\sigma_{1},\ldots,\sigma_{n-1}  \;
\begin{array}{|ll}
 \sigma_{i}\sigma_{j}\sigma_{i}=\sigma_{i}\sigma_{j}\sigma_{i}, &|i-j|=1 \\
\sigma_{i}\sigma_{j}=\sigma_{j}\sigma_{i}, & |i-j|>1 
\end{array}
\right\rangle \]

The elements of the generating set $\Sigma=\{\sigma_{1},\ldots,\sigma_{n-1}\}$ are called {\em Artin generators} or {\em Garside generators}.

In the dual Garside structure, we use a slightly different generating set which contains $\Sigma$ as a subset.
For $1\leqslant i < j \leqslant n$, let $a_{i,j}$ be the braid
\[ a_{i,j} = (\sigma_{i}\sigma_{i+1}\ldots \sigma_{j-2}) \sigma_{j-1} (\sigma_{i}\sigma_{i+1}\ldots \sigma_{j-2})^{-1}\] 

The generating set $\Sigma^{*} =\{a_{i,j}\: | \: 1\leqslant i < j \leqslant n\}$ was introduced in \cite{bkl}, and its elements are called the {\em dual Garside generators}, or {\em band generators}, or {\em Birman-Ko-Lee generators}.
 In our conventions, the braid $a_{i,j}$ is represented by the left-handed half Dehn-twist along an arc connecting $p_{i}$ and $p_{j}$ in the lower-half of the disc $\{z \in D^{2}\: | \: \textrm{Im}\,z <0 \}$.

A {\em dual-positive braid} is a braid which is written by a product of positive dual Garside generator $\Sigma^{*}$. The set of dual-positive braids is denoted by $B_{n}^{+*}$.
The {\em dual Garside element} is a braid $\delta$ given by
\[\delta= a_{1,2}a_{2,3}\ldots a_{n-1,n} \]

Let $\preccurlyeq_{\Sigma^{*}}$ be the subword partial ordering with respect to the dual Garside generating set~$\Sigma^{*}$: $\beta_{1} \preccurlyeq_{\Sigma^{*}} \beta_{2}$ if and only if $\beta_{1}^{-1}\beta_{2} \in B_{n}^{+*}$.
For a given braid $\beta$, the supremum $\sup_{\Sigma^{*}}(\beta)$ and the infimum $\inf_{\Sigma^{*}}(\beta)$ is defined by
\[ \sup\!{}_{\Sigma^{*}}(\beta) = \min \{ m \in \Z \: | \: \beta \preccurlyeq_{\Sigma^{*}} \delta^{m} \} \]
and 
\[ \inf\!{}_{\Sigma^{*}}(\beta) = \max \{ M \in \Z \: | \: \delta^{M} \preccurlyeq_{\Sigma^{*}} \beta \}\]
respectively.
A \emph{dual-simple} element is a dual-positive braid $s$ which satisfies $1 \preccurlyeq_{\Sigma^{*}} s \preccurlyeq_{\Sigma^{*}} \Delta$. The set of dual-simple element is denoted by $[1,\delta]$.
The Garside length $l_{\Sigma^{*}}$ is the length function with respect to the generating set $[1,\delta]$. 

The next formula relates the supremum, infimum and the length.

\begin{prop}[Dual Garside length]\label{prop:dGlength}
For a braid $\beta \in B_{n}$ we have the following equality
\[ l_{\Sigma^{*}}(\beta) = \max\{0,\sup\!{}_{\Sigma^{*}}(\beta)\} - \min\{ \inf\!{}_{\Sigma^{*}}(\beta),0\}. \]
\end{prop}

See \cite{bkl} for proof.

\subsection{The Lawrence-Krammer-Bigelow representation of the braid groups}
We review a definition of the Lawrence-Krammer-Bigelow representation. Our description of the LKB representation is homological and mainly follows Bigelow, but it is slightly modified so as to agree with our conventions. For details, see~\cite{b2}.

Let $C = \{(z_{1},z_{2}) \in D_{n} \times D_{n}\: | \: z_{1} \neq z_{2}\}\slash S_{2}$ be the configuration space of two unordered points of $D_{n}$.
We take base points $d_{1}=\exp(\frac{3}{2}\pi i-\varepsilon)$ and $d_2=\exp(\frac{3}{2}\pi i+\varepsilon)$ in $\partial D_{n}$ so that $d_2$ lies on the right side of $d_{1}$ -- see Figure~\ref{fig:fork}.
We take $\{d_{1},d_{2}\}$ as a base point of $C$.

Let $\langle q,t \rangle$ be the free abelian group of rank two generated $t$ and $q$, and define a homomorphism $\phi:\pi_{1}(C) \rightarrow \langle q,t \rangle$ as follows:
let $\{\gamma,\gamma'\}\co [0,1] \rightarrow C$ a loop representing an element $x \in \pi_{1}(C)$.
We define the number~$a$ as the sum of the winding numbers along each puncture $p_{j}$:
\[ a = -\frac{1}{2 \pi i} \sum_{j=1}^{n} \left(\int_{\gamma}\frac{dz}{z-p_{j}} + \int_{\gamma'}\frac{dz}{z-p_{j}}\right). \]
We define the number~$b$ as twice the winding number of the path $\gamma-\gamma'$ (i.e.\ as twice the relative winding number of the points):
\[ b = -\frac{1}{\pi i}  \int_{\gamma-\gamma'}\frac{dz}{z}\] 
Now we define $\phi(x)= q^{a}t^{b}$.
(Here we add the minus signs since we adapted the convention that the positive winding direction is the clockwise direction.)

Let $\pi:\tC \rightarrow C$ be the covering of $C$ associated to $\textrm{ker}\, \phi$.
Fix the lift of the base point $\{d_{1},d_{2}\}$. By abuse of notation, we use the same symbol $\{d_{1},d_{2}\}$ to represent the base point both in $C$ and $\tC$.
Then $q$ and $t$ are regarded as deck translations and the second homology 
group $H_{2}(\tC;\Z)$ is a $\Z[q^{\pm 1},t^{\pm 1} ]$-module.

Let $Y$ be the $Y$-shaped graph shown in Figure~\ref{fig:fork}, having one distinguished external vertex $r$, two other external vertices $v_{1}$ and $v_{2}$, and one internal vertex~$c$. We orient the edges of $Y$ as shown in Figure \ref{fig:fork}.

\begin{figure}[htbp]
 \begin{center}
 \SetLabels
(0.05*.82)    $v_{1}$\\
(0.24*.82)   $v_{2}$\\
(0.1*0.42)  $c$\\
(0.13*0.06)   $r$\\
(0.41*0.01)   $d_{1}$\\
(0.46*0.01)   $d_{2}$\\
(0.8*0.01)   $d_{1}$\\
(0.85*0.01)   $d_{2}$\\
(0.38*0.5)   $F$\\
(0.38*0.25)   $F'$\\
(0.85*0.65)   $N$\\
\endSetLabels
\strut\AffixLabels{\includegraphics*[scale=0.5, width=100mm]{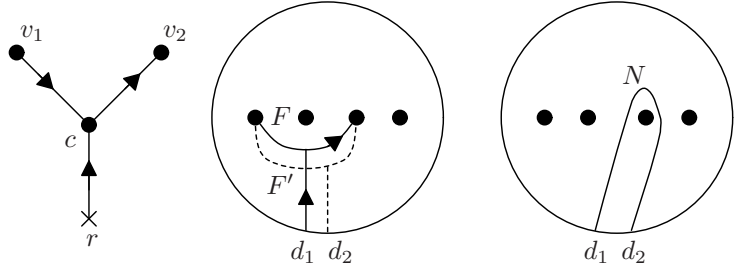}}
   \caption{Fork and Noodle}
 \label{fig:fork}
  \end{center}
\end{figure}

A {\em fork} $F$ is an embedded image of $Y$ into $D^2$ such that:
\begin{itemize}
\item All points of $Y\setminus \{r,v_1,v_2\}$ are mapped to the interior of~$D_n$.
\item The distinguished vertex $r$ is mapped to the base point $d_{1}$.
\item The other two external vertices $v_{1}$ and $v_{2}$ are mapped to two different puncture points.
\item The edge $[r,c]$ and the arc $[v_1,v_2]=[v_{1},c] \cup [c,v_{2}]$ are both mapped smoothly.
\end{itemize}

The image of the edge $[r,c]$ is called the {\em handle} of~$F$. The image of $[v_1,v_2]=[v_{1},c] \cup [c,v_{2}]$, regarded as a single oriented arc, is called the {\em tine}, denoted $T(F)$. The image of $c$ is called the \emph{branch point} of~$F$.
 
A \emph{parallel fork} $F'$ is a fork as depicted by a dotted line in Figure~\ref{fig:fork}: $F'$ is parallel to $F$ and the distinguished vertex is mapped to $d_{2}$. 

A {\em noodle} is an oriented smooth embedded arc which begins at $d_{1}$ and ends at $d_{2}$.

Let $\gamma, \gamma'\co [0,1] \rightarrow D_{n}$ be the handles of the forks~$F$ and of its parallel $F'$, respectively. Let $\{\widetilde{\gamma,\gamma'}\}\co [0,1] \rightarrow \tC$ be the lift of the path $\{\gamma,\gamma'\}\co [0,1] \rightarrow C$ taken so that $\{\widetilde{\gamma,\gamma'}\}(0)=\{d_{1},d_{2}\}$.

Consider the surface $\Sigma(F) =\left\{\{x,y\} \in C \: | \: x \in T(F), y \in T(F')\right\}$ in $C$. Let $\tSigma(F) \subset \tC$ be the component of $\pi^{-1}\Sigma(F)$ which contains the point $\{\widetilde{\gamma,\gamma'}\}(1)$. The surface $\widetilde{\Sigma}(F)$ in $\tC$ defines an element of $H_{2}(\tC;\Z)$. By abuse of notation, we will use the same symbol $F$ to represent the 2nd homology class defined by the surface $\tSigma(F)$.
In a similar way, a noodle $N$ defines a surface $\tSigma(N)$ which defines an element of $H_{2}(\tC,\partial \tC;\Z)$. Again, by abuse of notation we denote the homology class $[\tSigma(N)]$ by $N$.

For $1\leqslant i< j \leqslant n$, let us take a fork $F_{i,j}$ as shown in Figure \ref{fig:standfork}.
These forks are called {\em standard forks}. We call a standard fork of the form $F=F_{i,i+1}$ a {\em straight fork}. Similarly, let $N_{i}$ be the noodle which encloses the $i$-th puncture point $p_{i}$ as shown in Figure \ref{fig:standfork}. We call such a noodle a {\em standard noodle}.
Bigelow showed that $H_{2}(\tC;\R)$ is a free $\R[q^{\pm 1}, t^{\pm 1}]$-module and $\{F_{i,j}\}$ is a basis of $H_{2}(\tC;\R)$. From now on, by using this basis we always identify $H_{2}(\tC;\R)$ with $\R[q^{\pm 1},t^{\pm 1}]^{n(n-1)/2}$.
\begin{figure}[htbp]
 \begin{center}
 \SetLabels
(0.11*0.6)    $p_{i}$\\
(0.26*0.21)    $F_{i,j}$\\
(0.31*0.6)    $p_{j}$\\
(0.75*0.4)    $p_{i}$\\
(0.77*0.65)    $N_{i}$\\
\endSetLabels
\strut\AffixLabels{\includegraphics*[scale=0.5, width=72mm]{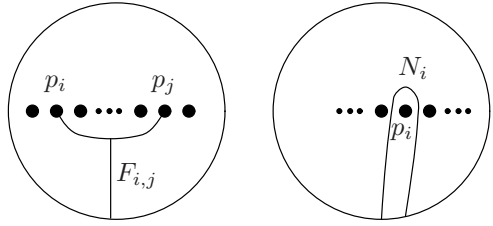}}
   \caption{Standard forks and standard noodles}
 \label{fig:standfork}
  \end{center}
\end{figure}

The braid group $B_{n}$ acts on the covering $\tC$ so that it commutes with the deck translations $t$ and $q$. Hence we get a representation $B_{n} \rightarrow \GL(H_{2}(\tC;\R)) = \GL(\frac{n(n-1)}{2}; \R[q^{\pm 1},t^{\pm 1}])$. The exact matrices are given in the Introduction. This is the {\em Lawrence-Krammer-Bigelow (LKB)} representation.

\begin{remark}
As shown in \cite{PaPa}, the standard forks do \emph{not} form a basis of the $\Z[q^{\pm 1},t^{\pm 1}]$-module $H_{2}(\tC;\Z)$. 
Thus in the above description it was important to use real coefficients, even if the matrix coefficients actually all lie in $\Z[q^{\pm 1},t^{\pm 1}]$.
\end{remark}

\subsection{Noodle-fork pairings}
\label{sec:pairing}
The Noodle-Fork pairing is a homology intersection pairing $H_{2}(\tC;\partial \tC) \times H_{2}(\tC) \rightarrow \Z[q^{\pm 1},t^{\pm 1}]$.
As Bigelow showed in his so-called Basic Lemma (\cite[Lemma 2.3]{b2}), the pairing is calculated as follows.

Let $z_{1},\ldots, z_{m}$ be the intersection points of $T(F)$ with $N$, and let $z'_{i}$ be the intersection of $T(F')$ with $N$ which corresponds to $z_{i}$. 

Observe that a pair of intersection points $\{z_{i},z_{j}'\} \in \tC$ corresponds to an intersection point of the surfaces $\widetilde{\Sigma}(N)$ and $\widetilde{\Sigma}(F)$. Hence it contributes to the total intersection pairing of $N$ and $F$ as a monomial $\varepsilon_{i,j}m_{i,j}= \varepsilon_{i,j}q^{a_{i,j}}t^{b_{i,j}}$ where $\varepsilon_{i,j}$ denotes the sign of the intersection at $\{z_{i},z'_{j}\}$.

The monomial $\varepsilon_{i,j}m_{i,j}$ is computed as follows.
First we define $c_{i,j}$ by 
\[
c_{i,j}=
\left\{
\begin{array}{ll}
+1 & \textrm{ if } d_{2} \textrm{ and } z'_{j} \textrm{ belong to the same component of } N-z_{i}.\\
-1 & \textrm{ if } d_{1} \textrm{ and } z'_{j} \textrm{ belong to the same component of } N-z_{i}.
\end{array}
\right.
\]  

Take three paths $A, B$ and $C$ in $D_{n}$ as follows:
\begin{itemize}
\item $A$ is a path from $d_{1}$ to the branch point of~$F$ along the handle of~$F$.
\item $B$ is a path from the branch point to $z_{i}$ along the tine $T(F)$.
\item $C$ is a path from $z_{i}$ to $d_{k}$ along the noodle $N$. Here $k=1$ if $c_{i,j}=+1$ and $k=2$ if $c_{i,j}=-1$. In other words, $C$ goes along $N$ starting from $z_{i}$, choosing the direction so as to avoid $z'_{j}$.
\end{itemize}

Similarly, we take three paths $A', B'$ and $C'$ by

\begin{itemize}
\item $A'$ is a path from $d'_{1}$ to the branch point of~$F'$ along the handle of~$F'$.
\item $B'$ is a path from the branch point to $z'_{j}$ along the tine $T(F')$.
\item $C'$ is a path from $z'_{j}$ to $d_{k'}$ along the noodle $N$. Here $k'=2$ if $c_{i,j}=+1$ and $k'=1$ if $c_{i,j}=-1$.
\end{itemize}

Now the concatenation of the three paths $\{C,C'\}\{B,B'\}\{A,A'\}$ defines a loop~$l$ in $C$. The monomial $m_{i,j}$ is given by $\phi(l)$.
The sign of the intersection $\varepsilon_{i,j}$ is given by
\[ \varepsilon_{i,j} = - s s' c_{i,j}\]
where $s$ (respectively $s'$) is the sign of the intersection of $N$ and $T(F)$ (respectively $T(F')$) at $z_{i}$ (respectively $z'_{j}$).

In summary, the noodle-fork pairing is given by the following sum, where we recall that~$m$ denotes the number of intersections of $T(F)$ with~$N$:
\[ \langle N,F\rangle = \sum_{1\leqslant i,j\leqslant m} \varepsilon_{i,j}m_{i,j} \in \Z[q^{\pm 1}, t^{\pm 1}]. \]

By direct computations we observe the following, which will play an important role in the proof of main theorems.

\begin{figure}[htbp]
 \begin{center}
\SetLabels
(0.08*0.93) $z_{1}$\\
(0.07*0.81)  $z'_{1}$\\
(0.2*0.93) $z_{2}$\\
(0.21*0.81)  $z'_{2}$\\
(0.6*0.85)  $\langle N,F\rangle =qt+ 1-t-q^{-1}$\\
(0.02*0.61)  $\{z_{1},z'_{1}\}$\\
(0.4*0.47) $ { \left\{ \begin{array}{l} a_{1,1}=1 \\ b_{1,1}=1 \\ \varepsilon_{1,1} = 1 \end{array} \right. }$\\
(0.02*0.29)  $\{z_{2},z'_{1}\}$\\
(0.4*0.18) $ { \left\{ \begin{array}{l} a_{2,1}=0 \\ b_{2,1}=1 \\ \varepsilon_{2,1} = -1 \end{array} \right. }$\\
(0.55*0.61)  $\{z_{1},z'_{2}\}$\\
(0.92*0.47) $ { \left\{ \begin{array}{l} a_{1,2}=0 \\ b_{1,2}=0\\ \varepsilon_{1,2} = 1 \end{array} \right. }$\\
(0.55*0.29)  $\{z_{2},z'_{2}\}$\\
(0.92*0.18) $ { \left\{ \begin{array}{l} a_{2,2}= -1 \\ b_{2,2}=0 \\ \varepsilon_{2,2} = -1 \end{array} \right. }$\\
  \endSetLabels
\strut\AffixLabels{\includegraphics*[scale=0.5, width=100mm]{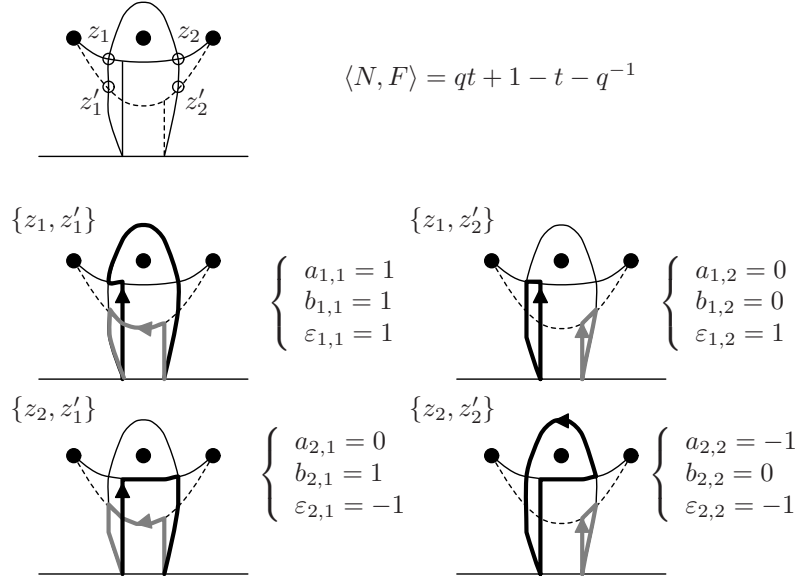}}
 \caption{Calculation of the noodle-fork pairing: an example}
\label{fig:example}
  \end{center}
\end{figure}

\begin{lem}
\label{lem:standpair}
Let $N$ be a standard noodle and let $F$ be a standard fork. Then  
\[ M_{q}(\langle N,F \rangle) \leqslant 1,\;\;\;\;m_{q}(\langle N,F \rangle) \geqslant -1. \]
\end{lem}
\begin{proof}
Here we give the calculation of the paring $\langle N,F\rangle$, for the most complicated case, the other cases are treated similarly.
Let $F= F_{i,j}$ and $N=N_{k}$ be a standard fork and noodle and assume that $i<k<j$. Then $T(F)$ and $N$ intersect at two points, $z_{1}$ and $z_{2}$ hence the two surfaces $\tSigma(N)$ and $\tSigma(F)$ intersects at four points. Now $a_{i,j},b_{i,j}$ and $\varepsilon_{i,j}$ are calculated as shown in Figure \ref{fig:example}. The paths $CBA$ and $C'B'A'$ are depicted by a black and gray line, respectively. Thus, we conclude $\langle N, F \rangle = qt+1-t-q^{-1}$. \end{proof}

\section{The wall crossing labeling of curve diagrams}

In this section we introduce the wall crossing labeling on curve diagrams and show that this labeling reflects the dual Garside length of braids. This result is interesting in its own right.

\subsection{Curve diagrams}

Let $\overline{E}$ be the diagram in $D_{n}$ consisting of the real line segment between the point $-1$ (the leftmost point of $\partial D^{2}$) and~$p_{n}$ (the rightmost puncture). Similarly, let~$E$ be the diagram in $D_{n}$ consisting of the real line segment between $p_{1}$ (the leftmost puncture) and $p_{n}$ (the rightmost puncture).
Both line segments $\overline{E}$ and $E$ are oriented from left to right.
Let $W_{i}$ be a vertical line segment in $D_{n}$, oriented upwards, which connects the puncture $p_{i}$ and the boundary of $D_{n}$ in the upper half-disk $\{z \in D^{2} \: | \: \textrm{Im}\, z>0 \: \}$ -- see Figure~\ref{fig:curvediagram}. The lines $W_i$ are called the {\em walls}, and their union $\bigcup W_{i}$ is denoted $W$.

\begin{figure}[htbp]
 \begin{center}
 \SetLabels
(0.12*0.65)    $\overline{E}$ and $E$\\
(0.43*0.7)    $W$\\
(0.88*0.35)    $\sigma_{1}(\overline{E})$\\
\endSetLabels
\strut\AffixLabels{\includegraphics*[scale=0.5, width=110mm]{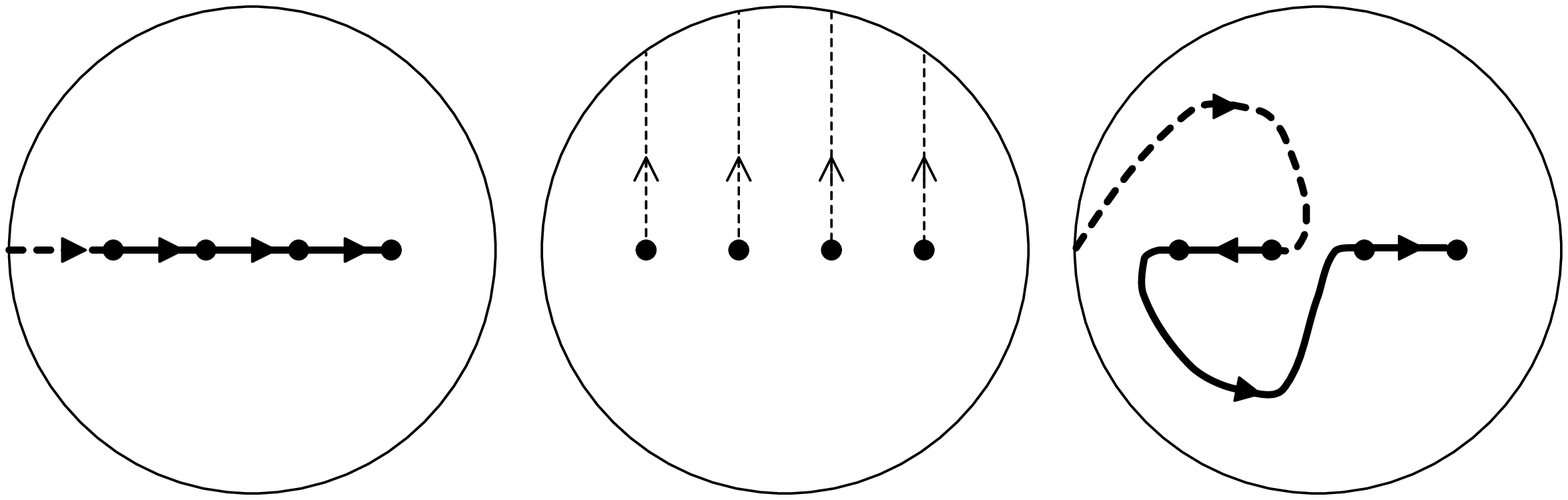}}
   \caption{Curve diagram and walls}
 \label{fig:curvediagram}
  \end{center}
\end{figure}

For $\beta \in B_{n}$, the {\em total curve diagram} and the {\em curve diagram} of $\beta$ is the image of the diagrams $\overline{E}$ and $E$, respectively, under a diffeomorphism $\phi$ representing $\beta$ which satisfies the following conditions.

\begin{enumerate}
\item The number of intersections of $\phi(\overline{E})$ with the walls is as small as possible within the diffeotopy class of~$\phi$.
\item Near the puncture points, the image $\phi(\overline{E})$ coincides with the real line.
\end{enumerate}

We denote the curve diagram of $\beta$ by $D_{\beta}$ and the total curve diagram by $\overline{D_\beta}$. See the right side of Figure \ref{fig:curvediagram} for an example.

To introduce the wall-crossing labeling and make the correspondence between fork and curve diagram explicit, we use a modified version of the curve diagrams.
For each puncture $p_{i}$ other than $\beta(p_{n})$, we take a small disc neighborhood of $p_{i}$, say~$B_{i}$, and let $B = \cup_{i}B_{i}$. Around each puncture $p_{i}$, we modify the curve diagram $D_{\beta}$ as shown in Figure \ref{fig:modcurve}. 
We denote the resulting (total) curve diagram by $MD_{\beta}$ $(\overline{MD_{\beta}})$, and call it the {\em (total) modified curve diagram}.
The right side of Figure \ref{fig:modcurve} shows the total modified curve diagram of $\sigma_{1}$.
\begin{figure}[htbp]
 \begin{center}
 \SetLabels
(0.13*0.72)    $p_{i}$\\
(0.2*0.8)    $B_{i}$\\
(0.13*0.27)    $p_{i}$\\
(0.2*0.34)    $B_{i}$\\
(0.86*0.33)    $\overline{MD_{\sigma_{1}}}$\\
\endSetLabels
\strut\AffixLabels{\includegraphics*[scale=0.5, width=110mm]{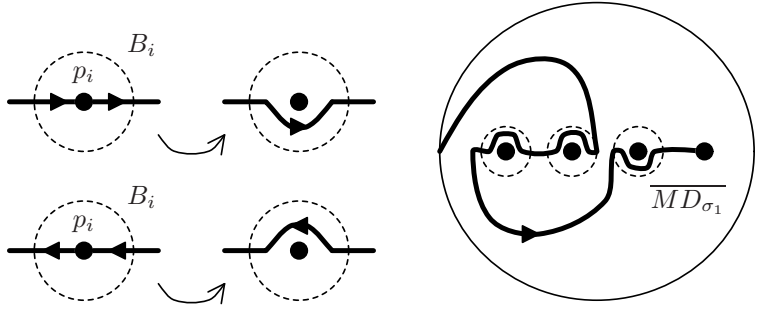}}
   \caption{Modified curve diagrams}
 \label{fig:modcurve}
  \end{center}
\end{figure}

Take a smooth parametrization of $\overline{MD_\beta}$, viewed as an image of the function $\gamma\co [0,1] \rightarrow D^{n}$. 
Then we define the {\em wall crossing labeling} as follows.

\begin{defn}
Let $R$ be the set of intersection points $\overline{MD_{\beta}} \cap W$. For each connected component $\alpha$ of $\overline{MD_{\beta}}-(R \cup (\overline{MD_{\beta}} \cap B))$, we assign the algebraic intersection number of $W$ and the arc $\gamma([0,s])$, where $s \in [0,1]$ is taken so that $\gamma(s) \in \alpha$. We call this integer-valued labeling $\Wcr(\alpha)$ the {\em wall crossing labeling}.
\end{defn}

An {\em arc segment} of the curve diagram $D_{\beta}$ (or the total curve diagram $\overline{D_{\beta}}$) is a component of $D_{\beta}-(D_{\beta}\cap (W \cup B))$ (or of $\overline{D_{\beta}}-(\overline{D_{\beta}}\cap (W \cup B))$, respectively). Since $\overline{MD_{\beta}}$ and $\overline{D_{\beta}}$ coincide except on $B$, an arc segment is identified with the subarc of $\overline{D_{\beta}}$. Using this correspondence, we assign the wall crossing labeling for each arc segment of the curve diagram -- see Figure~\ref{fig:labelex}.

\begin{figure}[htb]
\centerline{\input{labelex3.pstex_t}}
\caption{The total curve diagram of the braid $\beta=(\sigma_2^{-1}\sigma_1)^2$, and its wall crossing labeling. Among the labels of the solid (not dashed) arcs, the smallest one is $-2$ and the largest one is $2$.}
\label{fig:labelex}
\end{figure}

\begin{defn}\label{def:wcrlabel}
For a braid $\beta$, we define $\LWcr(\beta)$ and $\SWcr(\beta)$ as the largest and the smallest wall crossing number labelings occurring in the curve diagram $D_{\beta}$.
\end{defn}

Notice that in Definition~\ref{def:wcrlabel} we used the largest and smallest labels only of the curve diagram $D_{\beta}$, not the total curve digram $\overline{D_{\beta}}$. However, in order to determine the wall crossing labelings we need to consider the total curve diagram.

We now show that the dual Garside (Birman-Ko-Lee) length of a braid can be read off the wall crossing labeling of its curve diagram.

\begin{thm}
\label{thm:wall}
For a braid $\beta \in B_{n}$ we have the following equalities: 
\begin{enumerate}
\item $\sup_{\Sigma^{*}}(\beta) = \LWcr(\beta)$.
\item $\inf_{\Sigma^{*}}(\beta) = \SWcr(\beta)$.
\item $l_{\Sigma^{*}}(\beta) = \max( \LWcr(\beta),0) - \min( \SWcr(\beta),0 )$.
\end{enumerate}
\end{thm}

\begin{exam}
Let us consider the braid $\beta=(\sigma_2^{-1}\sigma_1)^2=(a_{2,3}^{-1}a_{1,2})^2$ -- see Figure~\ref{fig:labelex}. Any word representing this braid with letters belonging to $\{a_{1,2}^{\pm 1}, a_{2,3}^{\pm 1}, a_{1,3}^{\pm 1},\delta^{\pm 1}\}$ has at least two negative letters, so $\inf_{\Sigma^*}(\beta)=-2$. Similarly, any word representing $\beta$ using these letters has at least two positive letters, so $\sup_{\Sigma^*}(\beta)=2$. (Indeed, the dual normal form of $\beta$ is $a_{1,2}^{-1}.a_{1,3}^{-1}.a_{1,2}.a_{1,2}$.)

Now Theorem~\ref{thm:wall} asserts that the smallest and largest wall crossing labelings occurring in the curve diagram of $\beta$ should be $-2$ and~$2$, respectively, and Figure~\ref{fig:labelex} shows that this indeed the case.
\end{exam}

\begin{proof}[Proof of Theorem \ref{thm:wall}]

First of all, we show that for a dual-positive braid $\beta$ the equality $\LWcr(\beta)=l_{\Sigma^{*}}(\beta)$ holds.

To treat the dual Garside structure, we temporarily isotope our curve diagram and walls so that the all punctures sit on the circle $|z|=\frac{1}{2}$ and walls are disjoint from the subdisc $|z| \leqslant\frac{1}{2} $ as shown on the left side of Figure~\ref{fig:dualbraid}. This isotopy does not affect the wall crossing labelings, since the wall crossing labeling is defined by using algebraic intersection of arcs and walls.

As shown in \cite{bkl}, the set $[1,\delta]$ of dual-simple braids is in bijection with the set of disjoint collections of convex polygons in~$D_{n}$ whose vertices are punctures. This bijection is given as follows: to any such collection of polygons we can associate a dance of the puncture points, moving each puncture which belongs to some polygon~$P$ in the clockwise direction along the boundary of~$P$, to the position of the adjacent vertex. In this way, each dual-simple element can be represented by some disjoint convex polygons in $D_{n}$ -- see Figure~\ref{fig:dualbraid}.

We remark that polygons may be degenerate, having only two vertices; in the associated braids, the two corresponding punctures are interchanged by a clockwise half Dehn-twist. 

\begin{figure}[htbp]
 \begin{center}
 \SetLabels
(0.135*0.82)    $p_{1}$\\
(0.2*0.66)    $p_{2}$\\
(0.115*0.63)    $p_{n}$\\
(0.67*0.01)    $P$\\
(0.45*0.6) $P_1$\\
(0.55*0.42) $P_2$\\
\endSetLabels
\strut\AffixLabels{\includegraphics*[scale=0.5, width=110mm]{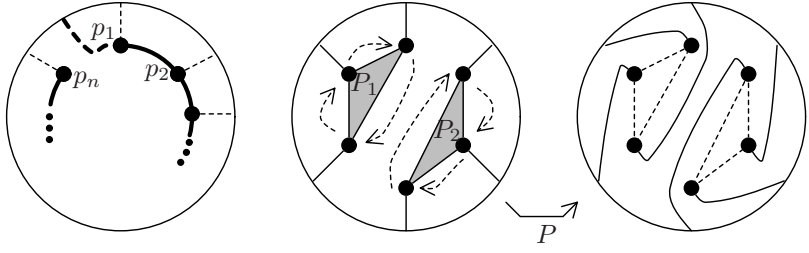}}
   \caption{The left picture shows the basic curve diagram (bold lines) and the walls (thin dashed lines). The pictures in the middle and on the left illustrate the action of a dual-simple braid~$P$ given by two disjoint polygons~$P_1$ and~$P_2$.}
 \label{fig:dualbraid}
  \end{center}
\end{figure}

Given a curve diagram $D_\beta$, equipped with the wall crossing labeling, and given a collection $P=\{P_1,\ldots,P_k\}$ of disjoint polygons with vertices in the punctures representing some dual-simple braid~$\beta^+$, let us describe in detail how to obtain the curve diagram $D_{\beta^+\beta}$ and its wall crossing labeling.
Consider the disjoint collection of annuli $A_1,\ldots,A_k$ in $D_n$ as follows (see Figure~\ref{fig:Paction}):
\begin{enumerate}
\item\label{descr:nopunctures} The outer and inner boundary component of $A_i$ are both homotopic in~$D_n$ to the boundary of a regular neighborhood of~$P_i$.
\item\label{descr:minpos} The boundary components of $A_i$ are in reduced relative position (no bigons) with respect to the walls $W$ and also with respect to the curve diagram~$D_\beta$.
\item  Among all collections of annuli satisfying (\ref{descr:nopunctures}) and  (\ref{descr:minpos}), we choose the one where the annuli contain as many intersection points $D_\beta\cap W$ as possible. Roughly speaking, we push as much twisting of the curve diagram around the polygons as possible into the annuli.
\end{enumerate}
Let us denote the components of $D_n\setminus A_i$ containing the polygon $P_i$ by $N(P_i)$. Note that the intersection of $D_\beta$ with $A_i$ is simply a spiral, and that the labels on this spiral interpolate linearly between the labels on the outer and inner boundary component.

Now the $\beta^+$-action on $D_\beta$ is simple to describe: on $D_n\setminus (\cup_{i=1}^{k} (A_i\cup N(P_i)))$, the diagram and its labeling is unchanged. On $N(P_i)$ the diagram is turned one notch in the clockwise sense, and all labels are increased by one. On the annuli, we have some 
twisting, but the labels still just interpolate - see Figure~\ref{fig:Paction}.

\begin{figure}[htbp] 
 \begin{center}
 \SetLabels
(0.43*0.62) $4$\\
(0.33*0.84) $3$\\
(0.2*0.62) $3$\\
(0.16*0.46) $2$\\
(0.3*0.23) $1$\\
(0.35*0.69) $0$\\
(0.4*0.59) $1$\\
(0.2*0.4) $P_i$\\
(0.07*0.58) $A_i$\\
(0.973*0.62) $4$\\
(0.89*0.84) $3$\\
(0.78*0.46) $4$\\
(0.71*0.50) $3$\\
(0.69*0.16) $2$\\
(0.84*0.23) $1$\\
(0.88*0.70) $0$\\
(0.95*0.59) $1$\\
(0.5*0.65) $\beta^+$\\
\endSetLabels
\strut\AffixLabels{\includegraphics*[scale=0.5, width=110mm]{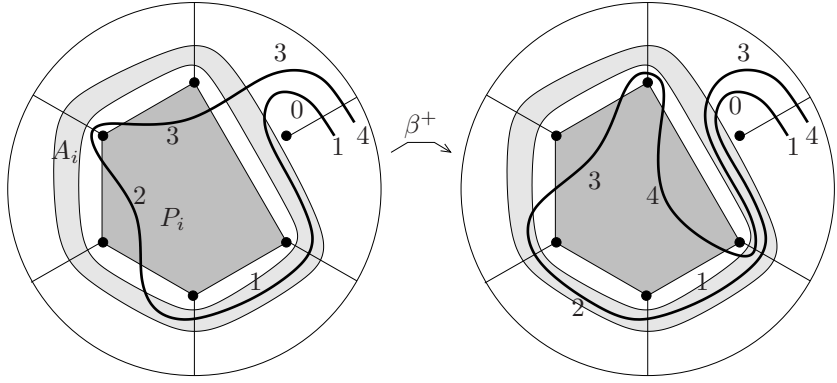}}
   \caption{Action of dual-simple elements}\label{fig:Paction}
  \end{center}
\end{figure}

The action of a \emph{negative} dual-simple braid~$\beta^-$ is similar, the only difference being that the twisting is in the counterclockwise direction and labels are decreased by one.

From the above description
for a dual-positive braid $\beta$ and a dual-simple element~$\beta^{+}$, we have the inequality $\LWcr(\beta^{+}\beta) \leqslant \LWcr(\beta) +1$. This implies the inequality 
\[ \LWcr(\beta) \leqslant l_{\Sigma^{*}}(\beta) \]
The converse inequality is now implied by the following lemma:

\begin{lemma}\label{lemma:shortproduct} The dual-positive braid~$\beta$ can be written as the product of $\LWcr(\beta)$ positive dual-simple braids.\end{lemma}

\begin{proof}[Proof of Lemma \ref{lemma:shortproduct}]
From the above description of the action of a dual-positive braid, we see that $\SWcr(\beta)\geqslant 0$, as no negative labels can ever be created from non-negative ones. 

In order to prove the lemma, we distinguish two cases. Firstly, if $\SWcr(\beta)=\LWcr(\beta)$, then no arc of the modified curve diagram $MD_\beta$ crosses a wall. This implies that $\beta$ is a power of $\delta$, more precisely, $\beta=\delta^{\SWcr(\beta)}$, and the lemma is true.

Secondly, if $0\leqslant\SWcr(\beta)<\LWcr(\beta)$, then we proceed inductively. We shall construct a negative dual-simple braid $\beta^{-}$
such that $\LWcr(\beta^{-}\beta)<\LWcr(\beta)$ and $\SWcr(\beta^{-}\beta)\geqslant \SWcr(\beta)$.

Consider the arc segments having the maximal wall crossing labeling.
Each such segment connects two walls.
Let $I=\{(i,j)\}$ be the pairs of walls which are connected by some maximal labeled arcs. Let $P$ be the minimal (with respect to inclusion) collection of convex polygons which contains all straight lines connecting $p_{i}$ and $p_{j}$ for $(i,j) \in I$. 
Let $\beta^{-}$ be the inverse of the dual-simple braid that corresponds to~$P$. 
Garside theoretically speaking, $\beta^{-} = (\bigvee_{(i,j)\in I} a_{i,j})^{-1}$.

According to our description above of the $\beta^-$-action on the curve diagram $D_\beta$, the action of $\beta^{-}$ decreases all the maximal wall crossing labelings by one, so $\LWcr(\beta^{-}\beta) = \LWcr(\beta) -1$. 

On the other hand, we claim that $\SWcr(\beta^{-}\beta) \geqslant \SWcr(\beta)$, i.e., contrary to the largest label, the smallest label does not decrease during the $\beta^-$-action. 

Let us prove this claim. First we observe that, roughly speaking, minimally labeled arcs are S-shaped, whereas maximally labeled arcs are \reflectbox{S}-shaped. More precisely, when both endpoints of a minimally labeled arc lie in the interior of walls, then the initial and terminal segment of the arc lie on the counterclockwise sides of the wall, whereas maximally labeled arcs begin and terminate on clockwise sides of the respective walls (see Figure~\ref{fig:nominlabel}(a)).

Now, in order to prove the claim, we have to rule out the existence of minimally labeled arcs (i.e.\ arcs labeled $\SWcr(\beta)$) which, under the $\beta^{-}$-action, give rise to arcs with an even smaller label.

First we observe that a minimally labeled arc~$\alpha$ cannot intersect the interior of any of the polygons of~$P$. Indeed, assume that an arc segment~$\alpha$ enters into the interior of one of these polygons, say $P_1$. Let us assume in addition that both endpoints of $\alpha$ lie in the interior of walls, not in punctures.
Then~$\alpha$ cuts $P_1$ into two components, and thus separates the vertex punctures of~$P_1$ into two families (drawn white and black in Figure~\ref{fig:nominlabel}(a)). Now we observe that no maximally labeled arc can connect a wall belonging to a black puncture to a wall belonging to a white puncture. But, by construction of~$P$, that means that the white and black punctures do not belong to the same polygon~$P_1$, which is a contradiction.

\begin{figure}[htbp]
 \begin{center}
 \SetLabels
(0.02*0.965) forbidden\\
(0.47*0.9) $\alpha$\\
(0.4*0.85) $\beta^{-}(\alpha)$\\
(0.4*0.965) allowed:\\
(0.94*0.965) forbidden\\
(0.31*0.12) Example of a\\
(0.32*0.01) max.\ labeled arc\\
(0.15*0.4) $\alpha$\\
(0*0.65) {\bf (a)}\\
(0.36*0.65) {\bf (b)}\\
(0.71*0.65) {\bf (c)}\\
\endSetLabels
\strut\AffixLabels{\includegraphics*[scale=0.5, width=110mm]{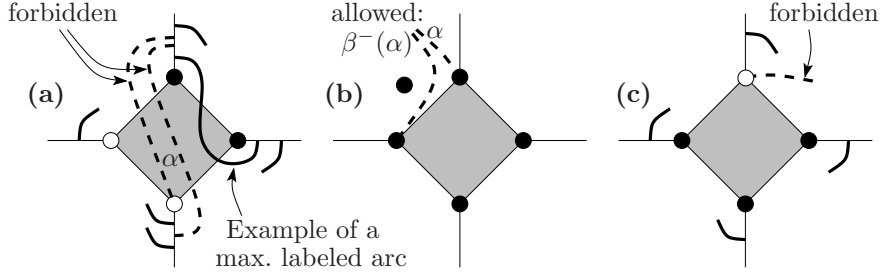}}
\caption{Fat dashed lines indicate minimally \Wcr-labeled arcs, and fat solid lines indicate maximally \Wcr-labeled arcs}
 \label{fig:nominlabel}
  \end{center}
\end{figure}

Similarly, when a minimally labeled arc~$\alpha$ intersects the interior of~$P$ but has one or both endpoints in punctures, then the same argument applies, we only have to decide in which color to paint a puncture at the extremity of~$\alpha$. The choice which works is to group such a puncture with the following punctures in the clockwise direction -- see Figure~\ref{fig:nominlabel}(a).

Finally, we have to consider a minimally labeled arc segment~$\alpha$ which ends in a vertex of~$P_1$ without intersecting the interior of~$P_1$. Here we have to distinguish two cases. If an extremal segment of $\alpha$ lies on the \emph{counterclockwise} side of the wall corresponding to its terminal puncture, as in Figure~\ref{fig:nominlabel}(b), then there is nothing to worry about since the $\beta^{-}$-action does not decrease its $\Wcr$-label. If, on the other hand, a terminal segment of~$\alpha$ lies on the \emph{clockwise} side of the wall corresponding to its terminal puncture, as in Figure~\ref{fig:nominlabel}(c), then this wall cannot be connected to a wall of any of the other punctures of~$P_1$ by a maximally labeled arc. Again, this contradicts the construction of~$P_1$.

In summary, no minimally labeled arc can generate an even smaller label under the $\beta^{-}$-action. This proves the claim, and thus Lemma~\ref{lemma:shortproduct}.
\end{proof}

To summarize, we have now proved that for a dual-positive braid $\beta$, the following equality holds:
\[ \LWcr(\beta) = l_{\Sigma^{*}}(\beta)\]
For a dual-positive braid $\beta$, $\inf_{\Sigma^{*}}(\beta)\geqslant 0$ so by Proposition \ref{prop:dGlength} we conclude that
\[  \LWcr(\beta) = \sup\!{}_{\Sigma^{*}}(\beta). \]
In a similar way, we prove the equality $ \SWcr(\beta) = \inf_{\Sigma^{*}}(\beta)$ for any dual-\emph{negative} braid~$\beta$.

In order to prove the same results for arbitrary braids, we recall that 
the dual Garside element $\delta$ acts as the
clockwise $\frac{2\pi}{n}$--rotation of 
the $n$-gon with vertices in all punctures. Thus the curve diagram $D_{\delta\beta}$ is obtained from $D_\beta$ simply by a $\frac{2\pi}{n}$--rotation,
and the wall crossing labeling on each arc segment of $D_{\delta\beta}$ is obtained from the label of the corresponding arc of $D_\beta$ by adding one. 
On the other hand, by definition of $\inf$ and $\sup$, left multiplication by $\delta$ increases both $\inf$ and $\sup$ by one.  Therefore for a general braid $\beta$, we get an equality
\begin{eqnarray*}
\LWcr(\beta) & = & \LWcr(\delta^{-\inf \! {}_{\Sigma^{*}}(\beta)}\beta) + \inf \! {}_{\Sigma^{*}}(\beta)\\
  & = & \sup \!{}_{\Sigma^{*}}(\delta^{-\inf \! {}_{\Sigma^{*}}(\beta)}\beta) + \inf \!{}_{\Sigma^{*}}(\beta)\\
 & = &  \sup \! {}_{\Sigma^{*}}(\beta).
\end{eqnarray*} 
A similar calculation also yields
\[ \SWcr(\beta) = \textstyle{\inf_{\Sigma^*}(\beta)}. \]
for a general braid $\beta$.
Finally, Proposition~\ref{prop:dGlength} implies the third equality claimed in Theorem~\ref{thm:wall}. This completes the proof of Theorem~\ref{thm:wall}.
\end{proof}

\begin{rem}
In \cite{w} the second author defined another labeling on the curve diagram called the {\em winding number labeling}, and proved the similar formula for the winding number labeling and the usual Garside length (\cite[Theorem 2.1]{w}). The proof of Theorem \ref{thm:wall} given here is a direct generalization of the proof of \cite[Theorem 2.1]{w}.
\end{rem}

\section{Noodle-fork pairing and wall crossing labeling}

In this section we make the observation (in Lemma~\ref{lem:Wcr=q}) that the wall crossing labeling reflects the exponents of the variable $q$ in the noodle-fork pairing.

We need some technical preparation for this. We first show that for any straight fork $F_{i,i+1}$, and for any braid~$\beta$, the tine $T(\beta(F_{i,i+1}))$ can in a natural way be equipped with the wall crossing labeling. For this, we will have to relate curve diagrams and forks.

Let us consider the part of the curve diagram $D_{\beta}$ that is the image of the line segment between the $i$-th and $(i+1)$-st punctures. We identify this part of the curve diagram with $\beta(T(F_{i,i+1}))$, the image of the tine of the straight fork $F_{i,i+1}$. Moreover, up to moving the footpoint from $-1$ to $d_1$, part of the modified curve diagram can naturally be regarded as the handle of $\beta(F_{i,i+1})$, as shown in Figure~\ref{fig:cdtofork}. This identification also induces the desired wall crossing labelling on each arc segment of $\beta(F_{i,i+1})$.

\begin{figure}[htbp]
 \begin{center}
 \SetLabels
(0.3*0.3)    $\overline{MD}_{\textsf{id}}$\\
(0.85*0.3)    $\textsf{id}(F_{3,4})$\\
(0.78*-0.05)    $d_{1}$\\
\endSetLabels
\strut\AffixLabels{\includegraphics*[scale=0.5, width=80mm]{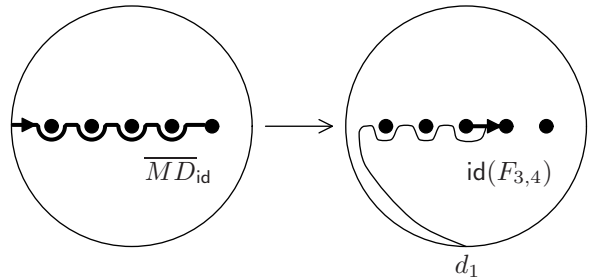}}
\caption{Viewing a curve diagram as a union of tines of forks, and viewing initial segments of modified curve diagrams as tines.}
\label{fig:cdtofork}
\end{center}
\end{figure}

Let $N$ be a standard noodle, and let $F$ be a straight fork with tine $T(F)$. Consider the image $\beta(F)$ of a straight fork~$F$. 
From now on, we always assume that~$N$ is isotoped so that it intersects $\beta(T(F))$ minimally.
For an intersection point $z_{i} \in \beta(T(F)) \cap N$, we denote the wall crossing labeling of the arc segment of $\beta(T(F))$ containing $z_{i}$ by $\Wcr(z_{i})$. We recall that the union of the walls is denoted~$W$. We also recall that the intersection pairing of~$N$ and~$\beta(F)$ is $\langle N,\beta(F)\rangle=\sum_{1\leqslant i,j\leqslant m} \varepsilon_{i,j}q^{a_{i,j}}t^{b_{i,j}}$ where the sum ranges over all pairs of intersection points $N\cap \beta(T(F))$.

\begin{lem}
\label{lem:Wcr=q}
Let $F$ be a straight fork and let $N$ be a standard noodle. Let $\{z_{i}\}$ be the set of intersection points of $\beta(T(F))$ and $N$. Then $a_{i,i}$, the exponent of $q$ at the intersection point $\{z_{i},z'_{i}\}$ is given as follows. 
\[ a_{i,i} =\left\{ \begin{array}{ll}
2\Wcr(z_{i})+1 & \textrm{if } z_{i} \textrm{ and } d_{1} \textrm{ belong to the same component of } N-W \\
2\Wcr(z_{i})-1 & \textrm{if } z_{i} \textrm{ and } d_{2} \textrm{ belong to the same component of } N-W.
\end{array}
\right.
\]
\end{lem}
\begin{proof}
Recall that $a_{i,i}$ is defined as the sum of the winding numbers of the paths $CBA$ and $C'B'A'$ around each puncture:
\[ a_{i,i} = - \frac{1}{2 \pi i} \sum_{j=1}^{n}\left( \int_{CBA}\frac{dz}{z-p_{j}} + \int_{C'B'A'}\frac{dz}{z-p_{j}}\right) \]
where $A,A',B,B',C,C'$ are the paths defined in Section \ref{sec:pairing}. See Figure \ref{fig:Wcr}.

\begin{figure}[htbp]
 \begin{center}
 \SetLabels
 (0.08*0.2)    $A$\\
(0.26*0.2)    $A'$\\
(0.07*0.68)    $B$\\
(0.07*0.42)    $B'$\\
(0.35*0.5)    $C$\\
(0.18*0.3)    $C'$\\
(0.18*0.68)    $z_{i}$\\
(0.18*0.42)    $z'_{i}$\\
(0.15*0.0)    $d_{1}$\\
(0.31*0.0)    $d_{2}$\\
 (0.62*0.2)    $A$\\
(0.78*0.2)    $A'$\\
(0.8*0.68)    $B$\\
(0.8*0.42)    $B'$\\
(0.62*0.5)    $C$\\
(0.89*0.3)    $C'$\\
(0.87*0.68)    $z_{i}$\\
(0.87*0.42)    $z'_{i}$\\
(0.66*0.0)    $d_{1}$\\
(0.83*0.0)    $d_{2}$\\
(0.2*0.84)    $W$\\
(0.72*0.84)    $W$\\
\endSetLabels
\strut\AffixLabels{\includegraphics*[scale=0.5, width=80mm]{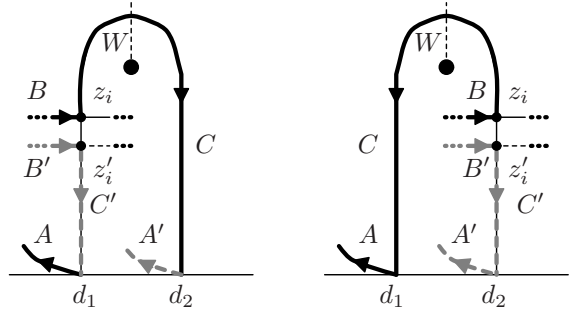}}
   \caption{Contribution of the noodle part}
 \label{fig:Wcr}
  \end{center}
\end{figure}

For a closed loop $\gamma$ in $D_{n}$, the winding number around the puncture $p_{j}$
\[ - \frac{1}{2 \pi i} \int_{\gamma} \frac{dz}{z-p_{j}}\]
is equal to the algebraic intersection number of $\gamma$ and $W_j$.
Thus, the fork part of the path $\{CBA,C'B'A'\}$, namely the subarc $\{BA,B'A'\}$, contributes to the total winding number $a_{i,i}$ by $2\Wcr(z_{i})$. Finally, we observe that the rest of the loop $\{C,C'\}$ (the noodle $N$ part of the path $\{CBA,C'B'A'\}$) contributes to 
the total winding number $a_{i,i}$ by $+1$ (respectively $-1$) if $z_{i}$ and $d_{1}$ belongs to the same (respectively to different) components of $N-(Z \cup W)$ -- see Figure~\ref{fig:Wcr}. This completes the proof.
\end{proof}

\begin{rem}
For the winding number labeling introduced in \cite{w}, and for $b_{i,i}$, the exponent of $t$ at the intersection $\{ z_{i},z'_{i} \}$, one can get a similar formula by similar arguments.  
Thus, schematically speaking, we have the following correspondence among three objects in braid groups:\\

\begin{tabular}[htbp]{|c|c|c|}
\hline
 LKB representation & Curve diagram &  Garside structure \\
\hline
\hline
Variable $t$ & Winding number labeling & Usual Garside structure \\
\hline
Variable $q$ & Wall crossing labeling & Dual Garside structure\\
\hline
\end{tabular}\\

Unfortunately, we did not manage to reprove Krammers result from \cite{k2}, i.e.\ the analogue of Theorem~\ref{thm:q} for the variable $t$ and the usual Garside length, using our geometric techniques.
\end{rem}

\section{The dual Garside length formula}

In this section we prove Theorem \ref{thm:q}.
For monomials $q^{a}t^{b}$ and $q^{a'}t^{b'}$ we define the lexicographical ordering $\leqslant_{q,t}$ by 
\[ q^{a}t^{b} \leqslant_{q,t} q^{a'}t^{b'} \textrm{ \ \ \ if }a < a', \textrm{ or if } a=a' \textrm{ and } b \leqslant b'.\]  
The next lemma is the crucial result in Bigelow's proof of faithfulness.

\begin{lem}\cite[Bigelow's Key Lemma 3.2 and Claim 3.4]{b2}
\label{lem:non-vanishing_q}
Assume that a noodle~$N$ and a fork $F$ have the minimal geometric intersection.
Then all intersection points $\{z_{i},z'_{j}\}$ of $N$ with $F$ which attain the $<_{q,t}$-maximal monomial $m_{i,j}$ in $\langle N,F \rangle$ have the same sign $\varepsilon_{i,j}$. 
\end{lem}

Roughly speaking, Lemma~\ref{lem:non-vanishing_q} states that the $<_{q,t}$-maximal contributions to $\langle N,F\rangle$ do not cancel.

\begin{lem}
\label{lem:Bigelow}
Let $\beta \in B_{n}$ be a braid. Let $F$ be a straight fork $F$ such that $\beta(T(F))$ contains an arc segment having the largest wall crossing labeling $\LWcr(\beta)$. Then there exists a standard noodle $N$ such that  
\[ M_{q}(\langle N, \beta(F) \rangle) \geqslant 2\LWcr(\beta) + 1. \]
\end{lem}

\begin{proof} 
Throughout proof, we assume that $\beta(T(F))$ intersects each standard noodle $N_{i}$ minimally.
It is sufficient to show that there is a standard noodle $N_{i}$ and an intersection point $z \in \beta(T(F)) \cap N_{i}$ such that $\Wcr(z)=\LWcr(\beta)$ and such that the two points~$z$ and~$d_{1}$ lie in the same component of $N_{i}-(W_{i} \cap N_{i})$: by Lemma \ref{lem:Wcr=q}, the intersection point $\{z,z'\}$ contributes to the pairing $\langle N_{i},\beta(F)\rangle$ by $\pm q^{2\LWcr(\beta)+1}t^{b}$ for some $b\in\Z$, and then Lemma~\ref{lem:non-vanishing_q} completes the proof.

\begin{figure}[htbp]
 \begin{center}
\SetLabels
(0.0*0.9) (a)\\
(0.03*0.3)  $z_{*}$\\
(0.05*0.0)  $d_{1}$\\
(0.13*0.0)  $d_{2}$\\
(0.14*0.82)  $N_{i}$\\
(0.2*0.9) (b)\\
(0.47*0.9) (c)\\
(0.75*0.9) (d)\\
(0.84*0.84) $N_{i}$\\
(0.97*0.84) $N_{i+1}$\\
(0.9*0.3) $z_{*}$\\
  \endSetLabels
\strut\AffixLabels{\includegraphics*[scale=0.5, width=110mm]{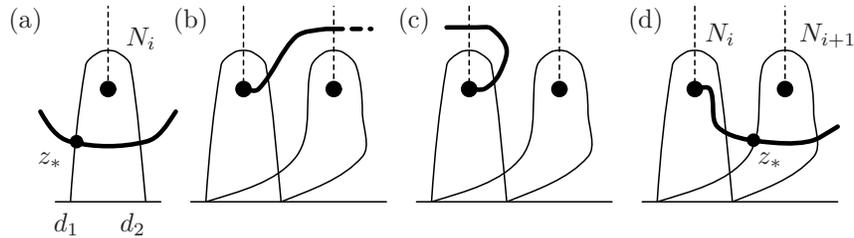}}
 \caption{Intersections achieving $q^{2 \LWcr(\beta)+1}$ exist}
\label{fig:intersection}
  \end{center}
\end{figure}

Let $\gamma$ be an arc segment of $\beta(T(F))$ whose wall crossing labeling is $\LWcr(\beta)$. Take a standard noodle $N_{i}$ 
which intersects $\gamma$ in a point~$z$. Assume that $z$ and $d_{1}$ lie on the two different components of $N_{i}-(W_{i} \cap N_{i})$.
If~$\gamma$ does not fall into the puncture~$p_{i}$, then we can find another intersection point $z_{*}$ of~$\gamma$ with~$N_i$ that lies on the same component of $N_{i}-(W_{i} \cap N_{i})$ as $d_{1}$ (Figure \ref{fig:intersection} (a)) and the proof is complete.
Assume that $\gamma$ falls into the puncture $p_{i}$. Then $\gamma$ must pass under the adjacent puncture $p_{i+1}$ because otherwise we either find another arc segment having strictly larger wall crossing labeling (Figure \ref{fig:intersection} (b)) or contradict the hypothesis that~$N_{i}$ and $\beta(T(F))$ have the minimal intersection (Figure \ref{fig:intersection} (c)). Then the standard noodle $N_{i+1}$ and $\gamma$ have an intersection $z_{*}$ with the desired property (Figure \ref{fig:intersection}(d)).
\end{proof}

\begin{proof}[Proof of Theorem \ref{thm:q}]
For a braid $\beta$, let us take a straight fork $F$ and a standard noodle $N$ as in Lemma \ref{lem:Bigelow}. Thus we have
$M_{q}(\langle N, \beta(F) \rangle) \geqslant 2 \LWcr(\beta) +1$.
Let us write the fork $\beta(F)$ as the linear combination of the standard forks
\[ \beta(F)=\sum_{1\leqslant i < j\leqslant n}x_{i,j} F_{i,j} \;\;\; (x_{i,j} \in \Z[q^{\pm 1},t^{\pm 1}]) .\]
Since $F$ is a straight fork, $x_{i,j}$ is an entry of the matrix $\mL(\beta)$.
Now we have an equality
\[ \langle N, \beta(F) \rangle = \sum_{1\leqslant i < j\leqslant n}x_{i,j} \langle N,F_{i,j} \rangle. \]
Since in Lemma \ref{lem:standpair} we observed that
\[ M_{q}(\langle N,F_{i,j}\rangle) \leqslant 1, \]
we conclude  
\[ \max_{1 \leqslant i < j \leqslant n} \{ M_{q}(x_{i,j})\} \geqslant
M_q(\langle N,\beta(F)\rangle)-1\geqslant
2\LWcr(\beta) = 2\sup\!{}_{\Sigma^{*}}( \beta). \]
where the second inequality follows from Lemma~\ref{lem:Bigelow} and the last equality from  Theorem~\ref{thm:wall}.
Thus, we get an inequality
\[ M_{q}(\mL(\beta)) \geqslant 2 \sup \!{}_{\Sigma^{*}}(\beta) .\]
On the other hand, for each dual-simple element $s \in [1,\delta]$
\[ M_{q}(\mL(s)) \leqslant 2 \]
holds. Hence we get the converse inequality
\[ M_{q}(\mL(\beta)) \leqslant 2 \sup\!{}_{\Sigma^{*}}(\beta). \] 
We conclude that
\[ M_{q}(\mL(\beta)) = 2 \sup\!{}_{\Sigma^{*}} (\beta)\]
The proof of (2) is similar, and (3) follows from Proposition~\ref{prop:dGlength}.
\end{proof}


\begin{thebibliography}{1}

\bibitem[Big]{b2}
S.\ Bigelow, {\em Braid groups are linear}, J. Amer. Math. Soc. \textbf{14}, (2000), 471--486.

\bibitem[Big2]{b3} S.\ Bigelow, {\em The Lawrence-Krammer representation}, 
Topology and geometry of manifolds (Athens, GA, 2001), 51--68,
Proc. Sympos. Pure Math., \textbf{71}, Amer. Math. Soc., Providence, RI, 2003.

\bibitem[BKL]{bkl} J.\ Birman, K.H.\ Ko, and S.J.\ Lee,
{\em{A new approach to the word problem in the braid groups,}}
Adv. Math. \textbf{139} (1998), 322--353.
 

\bibitem[Kra1]{k1}
D.\ Krammer, {\em The braid group $B_{4}$ is linear}, Invent. Math. \textbf{142}, (2000), 451--486.

\bibitem[Kra2]{k2}
D.\ Krammer, {\em Braid groups are linear}, Ann. Math. \textbf{155}, (2002), 131--156.


\bibitem[Law]{l}
R.\ Lawrence, {\em Homological representations of the Hecke algebra}, Comm.\ Math.\ Phys.\ \textbf{135}, (1990), 141--191.


\bibitem[PP]{PaPa} L.\ Paoluzzi, L.\ Paris, {\em A note on the Lawrence-Krammer-Bigelow representation}, Algebr.\ Geom. Topol.\ {\bf 2} (2002), 499--518

\bibitem[W]{w}
B.\ Wiest, {\em How to read the length from its curve diagram}, Groups Geom. Dyn. \textbf{5}, (2011), 673--681.
\end{thebibliography}
\end{document}